\documentclass[reqno]{amsart}
\usepackage{amsmath, amssymb, amsthm, epsfig}
\usepackage{hyperref, latexsym}
\usepackage{url}
\usepackage[mathscr]{euscript}
\usepackage{enumerate}
\usepackage{color}
\usepackage{fullpage} 
\usepackage{setspace}
\usepackage{cancel}

\def\today{\ifcase\month\or
  January\or February\or March\or April\or May\or June\or
  July\or August\or September\or October\or November\or December\fi
  \space\number\day, \number\year}

\DeclareMathOperator{\supp}{\mathrm{supp}}

 \newtheorem{theorem}{Theorem}
  
 \newtheorem{lemma}[theorem]{Lemma}
 \newtheorem{proposition}[theorem]{Proposition}
 \newtheorem{corollary}[theorem]{Corollary}
 \theoremstyle{definition}

 \theoremstyle{remark}

 \newcommand{\mc}{\mathcal}

 \newcommand{\M}{\mc{M}}

 \newcommand{\R}{\mathbb{R}}

 \newcommand{\hh}{\tfrac12}
 \newcommand{\ds}{\text{\rm d}s}

 \newcommand{\dt}{\text{\rm d}t}
  \renewcommand{\d}{\text{\rm d}}
 \newcommand{\du}{\text{\rm d}u}
 \newcommand{\dv}{\text{\rm d}v}

\newcommand{\im}{{\rm Im}\,}
\newcommand{\re}{{\rm Re}\,}

\begin{document}
\title[Large values of the argument of the\\ Riemann zeta-function and its iterates]{Large values of the argument of the\\ Riemann zeta-function and its iterates}
\author[Chirre and Kamal]{Andr\'{e}s Chirre and Kamalakshya Mahatab}

\address{Department of  Mathematical Sciences, Norwegian University of Science and Technology, NO-7491 Trondheim, Norway}
\email{carlos.a.c.chavez@ntnu.no }
\address{Kamalakshya Mahatab, Department of Mathematics and Statistics, University of Helsinki, P. O. Box 68, FIN
00014 Helsinki, Finland}
\email{accessing.infinity@gmail.com, \. kamalakshya.mahatab@helsinki.fi}

\thanks{AC was supported by Grant 275113 of the Research Council of Norway. KM was supported by Grant 227768 of the Research Council of Norway and Project 1309940 of Finnish Academy, and a part of this was work was carried out during his Leibniz fellowship at MFO, Oberwolfach.}

%\address{IMPA - Instituto Nacional de Matem\'{a}tica Pura e Aplicada - Estrada Dona Castorina, 110, Rio de Janeiro, RJ, Brazil 22460-320}
%\email{achirre@impa.br}

\allowdisplaybreaks
\numberwithin{equation}{section}

\maketitle

\begin{abstract}  
Let $S(\sigma,t)=\frac{1}{\pi}\arg\zeta(\sigma+it)$ be the argument of the Riemann zeta-function at the point  $\sigma+it$ in the critical strip. For $n\geq 1$ and $t>0$, we define
\begin{equation*}
S_{n}(\sigma,t) = \int_0^t S_{n-1}(\sigma,\tau) \,\d\tau\, + \delta_{n,\sigma\,},
\end{equation*}
where $\delta_{n,\sigma}$ is a specific constant depending on $\sigma$ and $n$. Let $0\leq \beta<1$ be a fixed real number. Assuming the Riemann hypothesis, we establish lower bounds for the maximum of $S_n(\sigma,t+h)-S_n(\sigma,t)$ near the critical line, on the interval $T^\beta\leq t \leq T$ and in a small range of $h$. This improves some results of the first author and generalizes a result of the authors on $S(t)$. We also give new omega results for $S_n(t)$, improving a result by Selberg.

%Our proof uses the resonator of Bondarenko and Seip with a new convolution formula for the variation of $S_{n}(\sigma,t)$ in short intervals.
\end{abstract}

\section{Introduction}
In this paper, we make use of the resonance method to improve several omega results related to the argument of the Riemann zeta-function. 

\subsection{The functions $S_n(\sigma,t)$}
Let $\zeta(s)$ denote the Riemann zeta-function. For $\hh\leq\sigma \leq 1$ and $t>0$, we define
$$S(\sigma,t) = \tfrac{1}{\pi} \arg \zeta \big(\sigma + it \big),$$
where the argument is obtained by continuous variation along straight line segments joining the points $2$, $2+i t$ and $\sigma + it$, assuming that the segment from  $\sigma + it$ to $2+i t$ has no zeros of $\zeta(s)$, and with the convention that $\arg \zeta(2) = 0$. If this path has zeros of $\zeta(s)$ (including the endpoint $\sigma + it$) we define $S(\sigma,t)  = \tfrac{1}{2}\, \lim_{\varepsilon \to 0} \left\{ S(\sigma,t + \varepsilon) + S(\sigma,t - \varepsilon)\right\}$. Let us define the iterates of the function $S(\sigma,t)$ in the following form: setting $S_{0}(\sigma,t):=S(\sigma,t)$, we define 
\begin{equation*}
S_{n}(\sigma,t) = \int_0^t S_{n-1}(\sigma,\tau) \,\d\tau\, + \delta_{n,\sigma\,} \text{ for } n\geq 1.
\end{equation*}
The constants $\delta_{n,\sigma}$ depends on $\sigma$ and $n$, and are given by 
$$\delta_{2k-1,\sigma} =\frac{ (-1)^{k-1}}{\pi} \int_{\sigma}^{\infty} \int_{u_{2k-1}}^{\infty} \ldots \int_{u_{3}}^{\infty} \int_{u_{2}}^{\infty} \log |\zeta(u_1)|\, \du_1\,\du_2\,\ldots \,\du_{2k-1}, $$
for $n = 2k-1$ with $k \geq 1$, and
$$\delta_{2k,\sigma} = (-1)^{k-1} \int_{\sigma}^{1} \int_{u_{2k}}^{1} \ldots \int_{u_{3}}^{1} \int_{u_{2}}^{1} \du_1\,\du_2\,\ldots \,\du_{2k} = \frac{(-1)^{k-1}(1-\sigma)^{2k}}{(2k)!},  $$ 
for $n = 2k$ with $k \geq 1$. 

\subsection{Large values on the critical line} In the case of $\sigma=\hh$, let us write $S_n(t)=S_n(\hh,t)$ to return to the classical notation (e.g. Littlewood \cite{L} and Selberg \cite{S2}). The argument function $S(t)$ is connected to the distribution of the non-trivial zeros of the Riemann zeta function through the classical Riemann von-Mangoldt formula %\footnote{The notation $f = O(g)$ means $|f(t)| \leq C \,g(t)$ for some constant $C>0$ and $t$ sufficiently large. In the subscript we indicate the parameters in which such constant $C$ may depend on.}
$$N(t) = \frac{t}{2\pi} \log \frac{t}{2\pi} - \frac{t}{2\pi} + \frac{7}{8}  + S(t) + O\bigg( \frac{1}{t}\bigg),$$
where $N(t)$ counts (with multiplicity) the number of zeros $\rho = \beta + i \gamma$ of $\zeta(s)$ such that $0 < \gamma \leq t$ (zeros with ordinate $\gamma = t$ are counted with weight $\hh$). The behavior of the functions $S_n(t)$ encodes the oscillatory character of the function $S(t)$ and efforts have been made to establish precise estimates of these functions (see %\cite{BS}, 
%\cite{CCM1}, 
\cite{CChi}, \cite{CChiM}, \cite{Chi},
 %\cite{I1}, 
 %\cite{I2}, 
 \cite{F1}, \cite{L}, \cite{L2}, \cite{S2},
%\cite{Ma}, 
\cite{S1}, \cite{W}).  

\vspace{0.2cm}
The Riemann hypothesis (RH) states that all the non-trivial zeros of $\zeta(s)$ have real part $\hh$. The classical estimates for $S_n(t)$ under RH are due to Littlewood \cite{L}, with the bounds $S_n(t)=O({\log t}/{(\log\log t)^{n+1}})$. The most recent refinements of these bounds are due to Carneiro, Chandee and Milinovich \cite{CCM1} for $n=0$ and $n=1$, and due to Carneiro and the first author \cite{CChi} for $n\geq 2$ (see also \cite{CChiM}). 
 On the other hand, Selberg\footnote{\,\,\, In \cite[Pages 3 and 4]{S1}, Selberg commented that these omega results were not established explicitly by Littlewood but can be proved by the usual methods.} established, assuming RH, that
\begin{align} \label{21_52}
S_n(t)=\Omega_{\pm}\bigg(\dfrac{(\log t)^{1/2}}{(\log\log t)^{n+1}}\bigg),
\end{align}
for $n\geq0$. The cases $n=0$ and $n=1$ were improved by Montgomery \cite[Theorem 2]{M} and Tsang \cite[Theorem 5]{T}, under RH, respectively:
$$
S(t) = \Omega_{\pm}\bigg(\dfrac{(\log t)^{1/2}}{(\log\log t)^{1/2}}\bigg), \hspace{0.3cm} \mbox{and}  \hspace{0.3cm} S_1(t) = \Omega_{\pm}\bigg(\dfrac{(\log t)^{1/2}}{(\log\log t)^{3/2}}\bigg).
$$
Using a new version of the classical resonance method, Bondarenko and Seip \cite[Theorem 2]{BS}, under RH, refined the order of magnitude of these omega results, showing that
\begin{align}  \label{21_31}
S(t) = \Omega\bigg(\dfrac{(\log t)^{1/2}(\log\log\log t)^{1/2}}{(\log\log t)^{1/2}}\bigg), \hspace{0.3cm} \mbox{and}  \hspace{0.3cm} S_1(t) = \Omega_{+}\bigg(\dfrac{(\log t)^{1/2}(\log\log\log t)^{1/2}}{(\log\log t)^{3/2}}\bigg).
\end{align}
Extending the method of Bondarenko and Seip, the first author \cite[Corollary 3]{Chi} established, under RH that
\begin{equation} \label{21_32}
%\label{eq:aqui-le-mostramos-como-hacerle-la-llave-grande}
S_n(t) = \left\{
\begin{array}{ll}\vspace{0.3cm}
\Omega_{+}\bigg(\dfrac{(\log t)^{1/2}(\log\log\log t)^{1/2}}{(\log\log t)^{n+1/2}}\bigg),& \mathrm{if\ } n\,\equiv\, 1 \, (\mathrm{mod} \, 4),  \\
\Omega\bigg(\dfrac{(\log t)^{1/2}(\log\log\log t)^{1/2}}{(\log\log t)^{n+1/2}}\bigg), & \mathrm{otherwise.\ } \
\end{array}
\right.
\end{equation}
Using the resonator of Bondarenko and Seip along with suitable kernels and RH, the authors \cite[Theorem 1]{Ma} have improved the result of Montgomery on $S(t)$ by proving 
\begin{align} \label{21_33}
S(t) = \Omega_{\pm}\bigg(\dfrac{(\log t)^{1/2}(\log\log\log t)^{1/2}}{(\log\log t)^{1/2}}\bigg).
\end{align}

\smallskip

\subsection{Large values near the critical line} Our main purpose in this paper is to extend the previous results of $S_n(t)$ to the function $S_n(\sigma,t)$ near the critical line. We will start by establishing bounds for the extreme values of the differences $S_n(\sigma,t+h)-S_n(\sigma,t)$.
%For small values of $h$, this measures the variation of $S_n(\sigma,t)$ over short intervals.

\begin{theorem} \label{27_9_7:50am}
	Assume the Riemann hypothesis. Let $0<\beta<1$ be a fixed real number and $n\geq 0$ be a fixed integer. 
	Let $T>0$ be sufficiently large, $h\in [0,(\log\log T)^{-1}]$, and $\sigma\geq \hh$. Consider the following two cases:
\begin{enumerate}
 \item[(i)] either 
 \[n=0 \quad \text{ and } \quad \dfrac{1}{2} < \sigma\leq \dfrac{1}{2}+\dfrac{1}{\log\log T},\]
 \item[(ii)] or 
 \[n\geq 1 \quad \text{ and } \quad \dfrac{1}{2} \leq \sigma\leq \dfrac{1}{2}+\dfrac{1}{\log\log T}.\]
\end{enumerate}
Then\footnote{\,\,\, The notation $f\gg g$ means that there is a positive constant $c>0$ such that $f(x)\geq c\,g(x)$.}
		$$
		\displaystyle\max_{T^\beta\leq t\leq T} \delta_n\{ S_n(\sigma,t+h)-S_n(\sigma,t)\} \gg h\,\dfrac{(\log T)^{1/2}(\log\log\log T)^{1/2}}{ (\log\log T)^{n-1/2}},
		$$
where $\delta_n=\pm1$ if $n$ is odd, and $\delta_n= (-1)^{(n+2)/2}$ if $n$ is even.
\end{theorem} 

The particular case of $n=1$ and $\sigma=\hh$ in Theorem \ref{27_9_7:50am} is related to a result of Tsang \cite[Theorem 6]{T}. Assuming RH, he proved that
$$
\displaystyle\sup_{T\leq t \leq 2T}\pm\{S_1(t+h)-S_1(t)\}\gg h\,\dfrac{(\log T)^{1/2}}{(\log\log T)^{1/2}},
$$
for $h\in [0,(\log\log T)^{-1}]$ (see also \cite[p. 252]{BLM}). Also Theorem \ref{27_9_7:50am} allows us to obtain extreme values for the functions $S_n(\sigma,t)$, improving a result of the first author \cite[Theorem 2]{Chi} (which is a general form of \eqref{21_32}).

\begin{corollary}  \label{9_29_4:06pm}
Assume the Riemann hypothesis. Let $0<\beta<1$ be a fixed real number and $n\geq 0$ be a fixed integer. Let $T>0$ be sufficiently large and suppose that 
\begin{equation} \label{9_29_1:58pm}
\dfrac{1}{2} \leq \sigma\leq \dfrac{1}{2}+\dfrac{1}{\log\log T}.
\end{equation}
Then
	$$
	\displaystyle\max_{T^\beta\leq t\leq T} \delta_n\{ S_n(\sigma,t)\} \gg \dfrac{(\log T)^{1/2}(\log\log\log T)^{1/2}}{ (\log\log T)^{n+1/2}},
	$$
	where $\delta_n=\pm1$ if $n$ is even, and $\delta_n= (-1)^{(n+3)/2}$ if $n$ is odd.
\end{corollary}

The case $n=0$ in Corollary \ref{9_29_4:06pm} was also studied by Tsang \cite[Theorem 2 and p. 382]{T}. He proved under RH that
	$$
	\displaystyle\sup_{T\leq t \leq 2T}\pm S(\sigma,t)\gg \dfrac{(\log T)^{1/2}}{(\log\log T)^{1/2}},
	$$
	in the range \eqref{9_29_1:58pm}. Note that for $n\geq 0$ and  $\sigma=\hh$, we recover the results in \eqref{21_31},  \eqref{21_32} and \eqref{21_33}, and we give new conditional omega results for $S_n(t)$. This improves the estimate of Selberg \eqref{21_52} in several cases:
\begin{equation*}
%\label{eq:aqui-le-mostramos-como-hacerle-la-llave-grande}
S_n(t) = \left\{
\begin{array}{ll}\vspace{0.3cm}
\Omega_{\pm}\bigg(\dfrac{(\log t)^{1/2}(\log\log\log t)^{1/2}}{(\log\log t)^{n+1/2}}\bigg),      & \mathrm{if\ } n\,\equiv\, 0 \, (\mathrm{mod} \, 4) \,\,\,\mbox{or} \, \,\,n\,\equiv\, 2 \, (\mathrm{mod} \, 4),  \\ \vspace{0.3cm}
\Omega_{+}\bigg(\dfrac{(\log t)^{1/2}(\log\log\log t)^{1/2}}{(\log\log t)^{n+1/2}}\bigg),& \mathrm{if\ } n\,\equiv\, 1 \, (\mathrm{mod} \, 4),  \\
\Omega_{-}\bigg(\dfrac{(\log t)^{1/2}(\log\log\log t)^{1/2}}{(\log\log t)^{n+1/2}}\bigg),    & \mathrm{if\ } n\,\equiv\, 3 \, (\mathrm{mod} \, 4).\
\end{array}
\right.
\end{equation*}

\smallskip

On the other hand, using an argument of Fujii, we can obtain some of these omega results unconditionally, when $n\geq 3$. \begin{corollary} \label{21_322} Unconditionally, for $n\geq 3$ and $n\,\not\equiv\, 3 \, (\mathrm{mod} \, 4)$, we have that
	$$
	S_n(t)=\Omega_{+}\bigg(\dfrac{(\log t)^{1/2}(\log\log\log t)^{1/2}}{ (\log\log t)^{n+1/2}}\bigg).
	$$
\end{corollary}

\medskip

We remark that the case $n=0$ and $\sigma=\hh$ has not been explored in Theorem \ref{27_9_7:50am}. The following result considers this exceptional case, proving a similar result, but in a shorter range.
\begin{theorem} \label{27_9_7:50am2}
	Assume the Riemann hypothesis. Let $0<\beta<1$ be a fixed real number. Then 
	$$
	\displaystyle\max_{T^\beta\leq t\leq T} -\{ S(t+h)-S(t)\} \gg h\,(\log T)^{1/2}(\log\log T)^{1/2}(\log\log\log T)^{1/2},
	$$
	for $h\in \big[c\,(\log T)^{-1/2}(\log\log T)^{-1/2}(\log\log\log T)^{-1/2},(\log\log T)^{-1}\big]$, with some constant $c>0$. 
\end{theorem}

Theorem \ref{27_9_7:50am2} improves an estimate of Selberg (unpublished\footnote{\,\,\, Tsang proved this result of Selberg in \cite[Page 388]{T}.}), where he proved under RH that 
$$
\displaystyle\max_{T\leq t\leq 2T} \pm\{ S(t+h)-S(t)\} \gg (h\,\log T)^{1/2},
$$
for $h\in [(\log T)^{-1},(\log\log T)^{-1}]$. 

%Furthermore, Theorem \ref{27_9_7:50am2} give us the following lower bound
%Finally, under RH and for $\hh < \sigma <1$, we may consider the following function related with the argument of $\zeta(s)$, which is the derivative
%$$S'(\sigma,t) = \frac{1}{\pi}\, \re \frac{\zeta'}{\zeta}(\sigma + it).$$
%Using Theorem \ref{27_9_7:50am} with $n=0$ and the relation \eqref{9_29_4:14pm}, we obtain the following corollary.
%
%\begin{corollary}
	%Assume the Riemann hypothesis. Let $0<\beta<1$ be fixed. Then there is a positive constant $c>0$ such that 
	%$$
	%\displaystyle\max_{T^\beta\leq t\leq T} -\re\dfrac{\zeta'}{\zeta}\bigg(\dfrac{1}{2}+\dfrac{1}{\log\log T}+it\bigg) \geq c\,(\log T)^{1/2}(\log\log T)^{1/2}(\log\log\log T)^{1/2}.
	%$$
%\end{corollary}

\medskip

\subsection{Sketch of the proof} Our approach is motivated by the modified version of the resonator of Bondarenko and Seip given by the first author in \cite[Section 3]{Chi}, and the convolution formula obtained by the authors in \cite{Ma}. We start by obtaining certain convolution formulas for $\log \zeta(\sigma+i(t+h))-\log \zeta(\sigma+i(t-h))$ in a small range of $h$. These formulas contain suitable kernels that are completely positive or completely negative\footnote{\,\,\, We say that a function $f$ is completely positive (or completely negative) if $f(x)\geq 0$ (or $f(x)\leq0$) for $x\in\R$.}, and it allows us to pick large positive and negative values. The connection between $\log\zeta(s)$ and $S_n(\sigma, t)$ expresses the convolution formula as two finite sums, of which we must detect which one is the main term, depending on the parity of $n$ and the new parameters involved. Then, we use the resonator due to the first author to obtain estimates for the variation of $S_n(\sigma,t)$ near the critical line. In particular, we highlight that one of the main technical difficulties of this work, when compared to \cite{BS, Chi, Ma}, is in the analysis of the error terms. With a more delicate computation, we obtain the term $h$ in each of the error term that appears in the convolution formulas. Finally, the choice of suitable parameters give the necessary control on the length of the Dirichlet polynomial to apply \cite[Lemma 13]{BS}, and the control on the sign in front of the variation of $S_n(\sigma,t)$. 

We would like to remark that Bui, Lester, and Milinovich \cite{BLM} used the version of the resonance method of Soundararajan \cite{S} to give a new proof of the omega results of Montgomery \cite{M}, using the variation of $S_1(t)$ in short intervals. We refer to \cite{kamal18} for another application of the resonance method to show $\Omega_\pm$ results.

\smallskip

 Throughout this paper we will use the notation $\log_2T=\log\log T$ and $\log_3T=\log\log\log T$. The error terms that appears in each estimate may depend on $n$.

%\subsection{Notation} The symbols $\ll$, $O(\,\cdot\,)$, $o(\,\cdot\,)$ are used in the standard way. In Section \ref{sec:num}, to compute explicit constants, we use the notation $\alpha=O^*(\beta)$ to mean that $|\alpha|\leq \beta$. For a function $h\in L^1(\R)$, we define the Fourier transform of $h$ by $$\widehat{h}(\xi)=\int_{-\infty}^\infty h(y)e^{-2\pi i\xi y}dy.$$

\section{Convolution formulas}

In this section we will obtain certain convolution formulas for $S_n(\sigma,t+h)-S_n(\sigma,t-h)$, when $n\geq 1$, related to kernels that are completely positive or completely negative. We will need the following estimate for $\zeta'(s)/\zeta(s)$ to prove our convolution formulas.

%Let us start, establishing the following estimate for the logarithmic derivative of $\zeta(s)$.

 %$\frac{\zeta'}{\zeta}(s)$
 \begin{lemma}\label{prop}
Assume the Riemann hypothesis. Then for $\hh<\sigma\leq 3$ and for sufficiently large $t$, we have
  $$
 \int_{\sigma}^{3}\bigg|\dfrac{\zeta'}{\zeta}(\alpha+it)\bigg|\d\alpha \ll (1+|\log(2\sigma-1)|)(\log t).
 $$
\end{lemma}

\begin{proof} Clearly, the integral from $2$ to $3$ is bounded. On the other hand, by \cite[Theorem 9.6 (A)]{tit}, uniformly for $\hh\leq \Re(s)\leq 2$, we have
$$
\dfrac{\zeta'(s)}{\zeta(s)}=\displaystyle\sum_{|t-\gamma|\leq 1}\dfrac{1}{s-\rho} + O(\log t),
$$
where the sum runs over the zeros $\rho=\hh+i\gamma$ of $\zeta(s)$. We conclude our required upper bound by integrating the above expression of $\zeta'(s)/\zeta(s)$ from $\sigma$ to $2$ and using the fact that the number of zeros with ordinate in $[t-1,t+1]$ is $O(\log t)$.
\end{proof} 

To simplify the notation, we will write 
\begin{equation*}
\Delta_h\log\zeta(z)=\log\zeta(z+ih)-\log\zeta(z-ih),
\end{equation*}
and 
\begin{equation*}
\Delta_hS_n(\sigma,t)=S_n(\sigma,t+h)-S_n(\sigma,t-h).
\end{equation*}

%Inspired in a lemma introduced by Montgomery \cite{M}, we obtain the following result.
\begin{lemma}  \label{9_20_10:05am}
	Assume the Riemann hypothesis. Let $0<\beta<1$ be a fixed number. Let $\alpha>0$, $H\in \R$ and $T$ be sufficiently large. Then for $\hh< \sigma\leq 2$, $0\leq h\leq 1$ and $T^{\beta}\leq t\leq T\log T$, we have
	\begin{align*}
	\int_{-(\log T)^{3}}^{(\log T)^{3}}\Delta_h&\log\zeta(\sigma+i(t+u))\bigg(\dfrac{\sin\alpha u}{u}\bigg)^2e^{iHu}\du \\& =-\pi i\displaystyle\sum_{m=2}^{\infty}\dfrac{\Lambda(m)\,w_m(\alpha,H)\sin(h\log m)}{(\log m)\,m^{\sigma+it}} + O\bigg(\dfrac{h(1+|\log(2\sigma-1)|)e^{2\alpha + |H|}}{(\log T)^3}\bigg),
	\end{align*} 
	where %\footnote{From the definition of $w_m(\alpha,H)$, the sum on Lemma \ref{9_20_10:05am} is finite or empty.} 
	$w_m(\alpha,H)=\max\{0,2\alpha-|H-\log m|\}$ 
	for all $m\geq 2$, and $\Lambda(m)$ is the von-Mangoldt function\footnote{\,\,\, $\Lambda(m)$ is defined as $\log p$ if $m=p^k$ with $p$ a prime number and $k\geq 1$ an integer, and zero otherwise.}.
\end{lemma} 
\begin{proof}
Our proof closely follows \cite[Lemma 2]{Ma}, so we have skipped some of the details in the proof. Using the Perron's summation formula, we write
\begin{align} \label{9_20_9am}
\begin{split}
\dfrac{1}{2\pi i}&\int_{1-i\infty}^{1+i\infty}\Delta_h\log\zeta(\sigma+it+s)\bigg(\dfrac{e^{\alpha s}-e^{-\alpha s}}{s}\bigg)^2e^{Hs}\ds \\
& = \displaystyle\sum_{m=2}^\infty \dfrac{\Lambda(m)\,w_m(\alpha,H)}{(\log m)\, m^{\sigma+i(t+h)}}-\displaystyle\sum_{m=2}^\infty \dfrac{\Lambda(m)\,w_m(\alpha,H)}{(\log m)\, m^{\sigma+i(t- h)}}  = -2i\,\displaystyle\sum_{m=2}^\infty \dfrac{\Lambda(m)\,w_m(\alpha,H)\sin(h\log m)}{(\log m)\, m^{\sigma+it}}.
\end{split}
\end{align} 
Note that $T^{\beta-\varepsilon}\leq |t\pm h|\leq T(\log T)^{1+\varepsilon}$, for some $\varepsilon>0$. Since we assume RH, we can move the
path of integration in \eqref{9_20_9am} to lie on the following five paths:
\begin{align*} 
&L_1=\{1+iu:(\log T)^{3}\leq u <\infty\}, \hspace{1.42cm} L_2=\{v+i(\log T)^{3}:0\leq v\leq 1\}, \\
&L_3=\{iu:-(\log T)^{3}\leq u <(\log T)^{3}\}, \hspace{0.86cm} L_4=\{v-i(\log T)^{3}:0\leq v\leq 1\}, \\
&L_5=\{1+iu:-\infty< u \leq -(\log T)^{3}\}.
\end{align*}
For each $1\leq j\leq 5$, we define the integrals
$$
I_j=\dfrac{1}{2\pi i}\int_{L_j}\Delta_h\log\zeta(\sigma + it+s)\bigg(\dfrac{e^{\alpha s}-e^{-\alpha s}}{s}\bigg)^2e^{Hs}\ds.
$$
% Firstly we estimate $I_1$ (the estimate for $I_5$ is similar). Note that for $u\geq (\log T)^{3}$, using the representation in \eqref{9_25_1:27pm} and $|\sin(x)|\leq |x|$ for $x\in\R$ we get\footnote{The notation $f\ll g$ means that there is a positive constant $c>0$ such that $f(x)\leq c\,g(x)$. In the subscript we indicate the parameters in which such constant $c$ may depend on.}
% \begin{align} \label{9_30_11:19am2}
% \begin{split} 
% \big|\Delta_h\log\zeta(\sigma+1+i(t+u))
% \big|  &=\Bigg|\displaystyle\sum_{m=2}^\infty \dfrac{\Lambda(m)}{(\log m)\, m^{\sigma+1+i(t+u)}}\bigg(\dfrac{1}{m^{ih}}-\dfrac{1}{m^{-ih}}\bigg)\Bigg|\\
% & \ll\displaystyle\sum_{m=2}^\infty \dfrac{\Lambda(m)}{(\log m)\, m^{\sigma+1}}|\sin(h\log m)|  \leq h\,\displaystyle\sum_{m=2}^{\infty}\dfrac{\Lambda(m)}{m^{\sigma+1}}\ll h.
% \end{split}
% \end{align}
%Then,
It takes standard computations to show 
\begin{align*}
|I_1|, |I_5|
%&\ll \int_{(\log T)^{3}}^{\infty}\big|\Delta_h\log\zeta(\sigma+1+i(t+u))\big|\Bigg|\dfrac{e^{\alpha(1\pm iu)}-e^{-\alpha(1\pm iu)}}{1\pm iu}\Bigg|^2e^{H}\du\\ &\ll h\,e^{2\alpha+H}\int_{(\log T)^{3}}^{\infty}\dfrac{1}{u^2}\du  \leq
\ll \dfrac{h\,e^{2\alpha+|H|}}{(\log T)^{3}}.
\end{align*} 
Now we estimate $I_2$ and the estimate for $I_4$ is similar. Using Fubini's theorem and Lemma~\ref{prop}, we have  
\begin{align*}
|I_2|&\ll\int_{0}^{1}\big|\Delta_h\log\zeta(\sigma+v+i\big(t+(\log T)^{3})\big)\big|\left(\frac{e^{\alpha v}+e^{-\alpha v}}{|v+i(\log T)^3|}\right)^2e^{v |H|}\dv\\
& \ll \frac{e^{2\alpha +|H|}}{(\log T)^6}\int_{0}^{1}\int_{-h}^h\bigg|\dfrac{\zeta'}{\zeta}\big(\sigma +v + i(t+u+ (\log T)^{3})\big)\bigg|\du\,\dv \\
& \ll \frac{e^{2\alpha +|H|}}{(\log T)^6} \int_{-h}^h\int_{0}^{1}\bigg|\dfrac{\zeta'}{\zeta}\big(\sigma +v + i(t+u+ (\log T)^{3})\big)\bigg|\dv\,\du \\
&\ll \dfrac{h(1+ |\log(2\sigma-1)|)e^{2\alpha+|H|}}{(\log T)^{5}}.
\end{align*} 
% Using this estimate and similar computations from \eqref{9_30_11:22am}, we obtain that
% \begin{align*}
% I_2 & \ll  \dfrac{h(1+ |\log(2\sigma-1)|)e^{2\alpha+|H|}}{(\log T)^{5}}.
% \end{align*} 
Finally, the integral $I_3$ gives us the main term:
\begin{align*}
I_3& =\dfrac{1}{2\pi}\int_{-(\log T)^{3}}^{(\log T)^{3}}\Delta_h\log\zeta(\sigma+i(t+u))\bigg(\dfrac{e^{i\alpha u}-e^{-i\alpha u}}{iu}\bigg)^2e^{iHu}\du \\
& =\dfrac{2}{\pi}\int_{-(\log T)^{3}}^{(\log T)^{3}}\Delta_h\log\zeta(\sigma+i(t+u))\bigg(\dfrac{\sin\alpha u}{u}\bigg)^2e^{iHu}\du.
\end{align*}
%Therefore, combining the above result and the error terms in \eqref{9_20_9am} we get the desired result.

%The only new addition here is to show that \footnote{The notation $f\ll g$ means that there is a positive constant $c>0$ such that $f(x)\leq c\,g(x)$. In the subscript we indicate the parameters on which such constant $c$ may depend.}
%\[\int_{0}^{1}\big|\Delta_h\log\zeta(\sigma+v+i\big(t+(\log T)^{3})\big)\big|\ll  h (1+|\log(2\sigma-1)|)\log T\]
 %using Fubini's theorem and the estimate %for $\hh<\sigma\leq 2$ and for sufficiently large $t$ we have
%\[\int_{\sigma}^{3}\bigg|\dfrac{\zeta'}{\zeta}(\alpha+it)\bigg|\d\alpha \ll (1+|\log(2\sigma-1)|)(\log t).\] 
\end{proof} 
Before obtaining the required convolution formulas for the differences $S_n(\sigma,t+h)-S_n(\sigma,t-h)$ for $n\geq 1$, we need to establish the following connection between $\Delta_h\log\zeta(z)$ and $\Delta_hS_n(\sigma,t)$.

\begin{lemma} \label{2:30_18_4} Assume the Riemann hypothesis. Let $n\geq 1$ be a fixed integer, $\hh\leq \sigma \leq 1$, and $t,h\in\R$ such that $t\neq \pm h$. Then we have
	\begin{equation*} 
	\Delta_hS_n(\sigma,t) = \frac{1}{\pi} \,\,\im{\left\{\dfrac{i^{n}}{(n-1)!}\int_{\sigma}^{2}{\left(v-\sigma\right)^{n-1}\,\,\Delta_h\log\zeta(v+it)}\,\dv\right\}} + O(h).
	\end{equation*}
\end{lemma}
\begin{proof}
For $t \neq 0$, integration by parts on \cite[Lemma 6]{CChiM} gives 
	\begin{equation*}
	S_n(\sigma,t) = \frac{1}{\pi} \,\,\im{\left\{\dfrac{i^{n}}{(n-1)!}\int_{\sigma}^{\infty}{\left(v-\sigma\right)^{n-1}\,\log\zeta(v+it)}\,\dv\right\}}.
	\end{equation*}
	So we have 
	\begin{align*}
	\Delta_hS_n(\sigma,t) & = \frac{1}{\pi} \,\,\im{\left\{\dfrac{i^{n}}{(n-1)!}\int_{\sigma}^{2}{\left(v-\sigma\right)^{n-1}\,\Delta_h\log\zeta(v+it)}\,\dv\right\}} \\
	& \,\,\,\,\, \,+ \frac{1}{\pi} \,\,\im{\left\{\dfrac{i^{n}}{(n-1)!}\int_{2}^{\infty}{\left(v-\sigma\right)^{n-1}\,\Delta_h\log\zeta(v+it)}\,\dv\right\}},	
	\end{align*} 
	for  $t\neq \pm h$.
	Finally the bound on the error term follows from the following estimate
	\begin{align*} %\label{9_30_11:19am}
		\int_{2}^{\infty}{\left(v-\sigma\right)^{n-1}\,\big|\Delta_h\log\zeta(v+it)}\big|\dv &=\int_{2}^{\infty}\left(v-\sigma\right)^{n-1}\,\Bigg|\displaystyle\sum_{m=2}^\infty \dfrac{\Lambda(m)}{(\log m)\, m^{v+it}}\bigg(\dfrac{1}{m^{ih}}-\dfrac{1}{m^{-ih}}\bigg)\Bigg|\dv \\
	& \ll\int_{2}^{\infty}\left(v-\sigma\right)^{n-1}\,\displaystyle\sum_{m=2}^\infty \dfrac{\Lambda(m)}{(\log m)\, m^{v+1}}|\sin(h\log m)|\dv \\
	& \leq h\,\displaystyle\sum_{m=2}^{\infty}\dfrac{\Lambda(m)}{m^{\sigma+1}}\int_{2}^{\infty}\dfrac{\left(v-\sigma\right)^{n-1}}{m^{v-\sigma}}\,\dv \ll h.
	\end{align*}
 \end{proof}
Note that for $\delta\in\{-1,1\}$ and $\delta'\in\{-1,0,1\}$, the function
$$
x\mapsto 3\,\delta+2\,\delta'\sin(x) 
$$
is completely positive or completely negative.
\begin{proposition} \label{9_21_4:20pm} Assume the Riemann hypothesis. Consider the following two cases: 
\begin{itemize}
 \item [(i)] either we have
 \[n\geq 1 \quad \text{ and } \quad \hh\leq\sigma<1,\]
 \item[(ii)] or
 \[n=0 \quad \text{ and } \quad \hh <\sigma<1.\]
\end{itemize}
Let $\beta, \gamma, \delta, \delta'\in\R$ be fixed parameters such that $0<\beta<1$, $\delta \in\{-1, 1\}$, and we further consider $\gamma, \delta'$ in the following two cases:
\begin{itemize}
 \item[(i')] either
 \[0<\gamma\leq  \hh \quad \text{ and } \quad \delta'\in\{-1,1\},\] 
 \item[(ii')]or  
 \[  \hh<\gamma\leq 1 \quad \text{ and } \quad \delta'=0.\]
\end{itemize}
Then for sufficiently large $T$, $T^\beta\leq t\leq T\log T$\ and $0\leq h\leq 1$, we have
\begin{align} \label{22_9_29_12}
\begin{split}
&\int_{-(\log T)^{3}}^{(\log T)^{3}} \Delta_hS_n(\sigma,t+u)\bigg(\dfrac{\sin(\gamma u\log_2T)}{u}\bigg)^2\big(3\,\delta+2\,\delta'\sin(u\log_2T)\big)\du \\
& \,\,\,\,\,\,\,\,\,\,\,\,\,\,\,\,\,\,\,\,\,= \im\Bigg\{3i^{n+3}\,\delta\displaystyle\sum_{m=2}^{\infty}\dfrac{\Lambda(m)\,a_{m}(T,h)}{(\log m)^{n+1}m^{\sigma+it}}\Bigg\}+\im\Bigg\{i^{n+2}\,\delta'\displaystyle\sum_{m=2}^{\infty}\dfrac{\Lambda(m)\,b_{m}(T,h)}{(\log m)^{n+1}m^{\sigma+it}}\Bigg\} + O(h\log_2T),
\end{split}
\end{align}
where the functions $a_{m}(T,h)$ and $b_{m}(T,h)$ are defined by
\begin{equation*} %\label{15_49_31_12}
a_{m}(T,h)=w_m(\gamma\log_2T,0)\sin(h\log m) \hspace{0.3cm} \mbox{and} \hspace{0.3cm}
b_{m}(T,h)=w_m(\gamma\log_2T,\log_2T)\sin(h\log m).
\end{equation*}
%\textcolor{red}{ The equality in \eqref{22_9_29_12} also holds when $\hh<\gamma<1$ and $\delta'=0$.}
\end{proposition}
\begin{proof}
We apply Lemma~\ref{9_20_10:05am} with $\alpha=\gamma\log_2T$, and $H=0$, $H=\log_2T$ and $H=-\log_2T$. Using the linear combination 
$$
3\,\delta+2\,\delta'\sin(u\log_2T)=3\,\delta\, e^{0} -i\,\delta'\big(e^{iu\log_2T}-e^{-iu\log_2T}\big),
$$
we obtain
\begin{align} \label{9_20_11:23am}
\begin{split}
\int_{-(\log T)^{3}}^{(\log T)^{3}}&\Delta_h\log\zeta(v+i(t+u))\bigg(\dfrac{\sin(\gamma u\log_2T)}{u}\bigg)^2\big(3\,\delta+2\,\delta'\sin(u\log_2T)\big)\du\\
& = -3\pi i\,\delta\displaystyle\sum_{m=2}^{\infty}\dfrac{\Lambda(m)\,a_m(T,h)}{(\log m)\,m^{v+it}}-\pi\,\delta'\displaystyle\sum_{m=2}^{\infty}\dfrac{\Lambda(m)\,b_m(T,h)}{(\log m)\,m^{v+it}}+ O\bigg(\dfrac{h(1+|\log(2v-1)|)}{(\log T)^{2-2\gamma}}\bigg),
\end{split}
\end{align}
when $\hh<v<2$.
%where the functions $a_{m}(T,h)$ and $b_{m}(T,h)$ were defined above. 
Note that for $0<\gamma\leq \hh$, we have used that $w_m(\gamma\log_2T,-\log_2T)=0$ for all $m\geq 2$. When $\gamma>\hh$ and $\delta'=0$, only the first sum on the right-hand side of \eqref{9_20_11:23am} remains. To obtain the case $n=0$, we take the imaginary part in \eqref{9_20_11:23am}. When $n\geq 1$, we want to use Lemma \ref{2:30_18_4} in \eqref{9_20_11:23am}. For $\hh\leq \sigma <1$, using Fubini's theorem (justified by \cite[Eq. (2.13)]{T} and the fact that the sums involved in \eqref{9_20_11:23am} are finite) we get,
\begin{align}  \label{9_26_3:18pm}
\begin{split} 
&\int_{-(\log T)^{3}}^{(\log T)^{3}} \bigg\{\int_{\sigma}^{2}(v-\sigma)^{n-1}\,\Delta_h\log\zeta(v+i(t+u))\,\dv\bigg\}\bigg(\dfrac{\sin(\gamma u\log_2T)}{u}\bigg)^2\big(3\,\delta+2\,\delta'\sin(u\log_2T)\big)\du \\
& \,\,\,\,\,\,\,\,\,\,\,\,\,= -\displaystyle\sum_{m=2}^{\infty}\dfrac{\Lambda(m)}{(\log m)m^{it}}\Bigg[\big(3\pi i\,\delta a_m(T,h)+\pi\delta'b_m(T,h)\big)\int_{\sigma}^{2}\dfrac{(v-\sigma)^{n-1}}{m^{v}}\dv \Bigg] + O\bigg(\dfrac{h}{(\log T)^{2-2\gamma}}\bigg).
\end{split}
\end{align}
Using  \cite[\S 2.321 Eq. 2] {GR}, we obtain 
\begin{align*} 
\int_{\sigma}^{2}{\dfrac{(v-\sigma)^{n-1}}{m^{v}}\,}\,\dv = \dfrac{\beta_{n-1}}{m^\sigma (\log m)^n} - \dfrac{1}{m^{2}}\displaystyle\sum_{k=0}^{n-1}\dfrac{\beta_k}{(\log m)^{k+1}}(2-\sigma)^{n-1-k},
\end{align*}
where $\beta_k=\frac{(n-1)!}{(n-1-k)!}$. This implies that for each $m\geq 2$, we have
\begin{align*}
\int_{\sigma}^{2}\dfrac{(v-\sigma)^{n-1}}{m^{v}}\,\dv & = \dfrac{(n-1)!}{m^{\sigma}(\log m)^{n}} + O\bigg(\dfrac{1}{m^2\log m}\bigg).
\end{align*}
Inserting this in \eqref{9_26_3:18pm} and using the estimates $|a_m(T,h)|, |b_m(T,h)|\ll h\log m\log_2T$, it follows that
\begin{align*}  
\int_{-(\log T)^{3}}^{(\log T)^{3}} &\Bigg\{\int_{\sigma}^{2}\left(v-\sigma\right)^{n-1}\,\Delta_h\log\zeta(v+i(t+u))\,\dv\Bigg\}\bigg(\dfrac{\sin(\gamma u\log_2T)}{u}\bigg)^2\big(3\,\delta+2\,\delta'\sin(u\log_2T)\big)\du \\
& = -3\pi(n-1)!\,\delta i\displaystyle\sum_{m=2}^{\infty}\dfrac{\Lambda(m)\,a_m(T,h)}{(\log m)^{n+1}m^{\sigma+it}}-\pi(n-1)!\,\delta'\displaystyle\sum_{m=2}^{\infty}\dfrac{\Lambda(m)\,b_m(T,h)}{(\log m)^{n+1}m^{\sigma+it}}  + O(h\log_2T).
\end{align*}
Finally the proof follows by using Lemma \ref{2:30_18_4} and calculating the error terms.
\end{proof}

\section{The Resonator}
In this section we recall the resonator $|R(t)|^2$ developed in \cite[Section 3]{Chi}. Let
\begin{align}  \label{9_22_2:21pm}
R(t)=\displaystyle\sum_{m\in\mathcal{M}'}\dfrac{r(m)}{m^{it}},
\end{align}
and $\mathcal{M}'$ be a suitable finite set of integers. Let $\sigma$ be a positive real number and $N$ be a positive integer sufficiently large such that 
\begin{align} \label{1_5_9:20pm}
\dfrac{1}{2} \leq \sigma\leq \dfrac{1}{2}+\dfrac{1}{\log\log N}.
\end{align}
Let $\mathcal{P}$ be the set of prime numbers $p$ such that
\begin{align} \label{20_4_6:46pm}
e\log N\log_{2}N < p \leq \exp\big((\log_2N)^{1/8}\big)\log N\log_2N.
\end{align}
We define $f(n)$ as a multiplicative function supported on square-free numbers such that
$$
f(p):=\bigg({\dfrac{(\log N)^{1-\sigma}{(\log_2N)}^{\sigma}}{(\log_3N)^{1-\sigma}}}\bigg)\dfrac{1}{p^{\sigma}\,(\log p-\log_2N-\log_3N)}
$$
for $p\in \mathcal{P}$, and $f(p)=0$ otherwise. For each $k\in\big\{1, ...,\big[(\log_2N)^{1/8}\big]\big\}$, we define the following sets:
\begin{align*}  %\label{9_20_1:43pm}
P_k:=\big\{p\in \mathcal{P}:  \hspace{0.1cm} e^k\log N\log_2N<p\leq e^{k+1}\log N\log_2N\big\},
\end{align*}
\begin{align*}
M_k:=\bigg\{n\in\supp(f): n  \hspace{0.1cm} \mbox{has at least} \hspace{0.1cm} \alpha_{k}:=\frac{3(\log N)^{2-2\sigma}}{k^2(\log_3N)^{2-2\sigma}} \hspace{0.1cm} \mbox{prime divisors in}  \hspace{0.1cm} P_k\bigg\},
\end{align*}
and 
\begin{align*}
\mathcal{M}:=\supp(f) \backslash \bigcup_{k=1}^{[(\log_2N)^{1/8}]}M_k .
\end{align*}

\subsection{Construction of the resonator} \label{subsection}
Let $0\leq \beta<1$ and $\kappa=(1-\beta)/2$. Note that $\kappa+\beta<1$. Let
\begin{align}  \label{17_24}
\dfrac{1}{2}\leq \sigma\leq \dfrac{1}{2}+\dfrac{1}{\log\log T},
\end{align}
and $N=[T^{\kappa}]$, so that $\sigma$ and $N$ satisfy the relation \eqref{1_5_9:20pm}. Let $\mathcal{J}$ be the set of integers $j$ such that
$$
\Big[\big(1+T^{-1}\big)^{j},\big(1+T^{-1}\big)^{j+1}\Big)\bigcap \mathcal{M} \neq \emptyset,
$$
and we define $m_j$ to be the minimum of $\big[(1+T^{-1})^{j},(1+T^{-1})^{j+1}\big)\cap \mathcal{M}$ for $j$ in $\mathcal{J}$. Consider the set
$$
\mathcal{M}':=\{m_j:j\in\mathcal{J}\},
$$
and finally we define
$$
r(m_j):=\Bigg(\displaystyle\sum_{n\in\mathcal{M},(1+T^{-1})^{j-1}\leq n \leq (1+T^{-1})^{j+2}}f(n)^2\Bigg)^{1/2}
$$
for every $m_j\in\mathcal{M}'$. This defines our Dirichlet polynomial in \eqref{9_22_2:21pm}. 
\\

 Let $\Phi(t):=e^{-t^2/2}$. We collect the following results proved in \cite[Section 3]{Chi}.
\begin{proposition} \label{19_4_6:04pm}
With the notations as above, we have
\begin{enumerate}[(i)]
\item $|R(t)|^2 \leq R(0)^2\ll T^{\kappa}\displaystyle\sum_{l\in\mathcal{M}}f(l)^2$,
\item $
\int_{-\infty}^{\infty}|R(t)|^{2}\,\Phi\bigg(\dfrac{t}{T}\bigg)\,\dt \ll T\displaystyle\sum_{l\in\mathcal{M}}f(l)^2.
$
\end{enumerate}
\end{proposition}
\begin{proof}
See \cite[Proposition 11]{Chi} and \cite[Lemma 12]{Chi}. 
\end{proof}

\begin{lemma} \label{20_41_1_20}
Suppose 
	$$
	G(t):=\displaystyle\sum_{n=2}^\infty\dfrac{\Lambda(n)\,c_n}{(\log n)n^{\sigma+it}}
	$$	
	is absolutely convergent and $c_n\geq 0$ for $n\geq 2$. Then
	$$	
	\int_{-\infty}^{\infty}G(t)|R(t)|^{2}\,\Phi\bigg(\dfrac{t}{T}\bigg)\,\dt \gg T\,\dfrac{(\log T)^{1-\sigma}(\log_3T)^{\sigma}}{(\log_2T)^{\sigma}}\bigg(\min_{p\in\mathcal{P}}c_p\bigg)\displaystyle\sum_{l\in\mathcal{M}}f(l)^2.
	$$
\end{lemma}
\begin{proof}
	See \cite[Lemma 13]{Chi}.
\end{proof}

The following result allows us to obtain the error terms in our theorems.

\begin{lemma}\label{9_22_6:02pm2} Assume the Riemann hypothesis, and consider the parameters defined in Proposition \ref{9_21_4:20pm}. Then
%, if $\gamma>0$ we have
	$$\Bigg|\displaystyle\sum_{m=2}^\infty\dfrac{\Lambda(m)\,a_{m}(T,h)}{(\log m)^{n+1}m^{\sigma}}\bigg(\int_{0}^{\infty}m^{-it}|R(t)|^2\Phi\bigg(\dfrac{t}{T}\bigg)\dt\bigg)\Bigg| \ll h\,T\,\dfrac{(\log T)^{2\gamma(1-\sigma)}}{(\log_2T)^{n-1}} \displaystyle\sum_{l\in\M}f(l)^2.$$
%	Note that $a_{m}(T,h)=0$ for all $m > (\log T)^{2\gamma}$, so the sum in the left-hand side of the equation is finite.
\end{lemma}
\begin{proof} Using the estimate $|a_m(T,h)|\ll h\log m\log_2T$, the fact that the sum runs over $2\leq m \leq (\log T)^{2\gamma}$, and $(ii)$ of Proposition \ref{19_4_6:04pm} it follows that
	\begin{align*}
	\Bigg|\displaystyle\sum_{m=2}^{\infty}\dfrac{\Lambda(m)\,a_{m}(T,h)}{(\log m)^{n+1}m^\sigma}\bigg(\int_{0}^{\infty}m^{-it}|R(t)|^2\Phi\bigg(\dfrac{t}{T}\bigg)\dt\bigg)\Bigg| \ll h\,T\log_2T\displaystyle\sum_{2\leq m \leq (\log T)^{2\gamma}}\dfrac{\Lambda(m)}{(\log m)^{n}m^\sigma}\displaystyle\sum_{l\in\mathcal{M}}f(l)^2.
	\end{align*}
	Using the prime number theorem (see \cite[B.1 Appendix]{CChiM}), we have
	\begin{equation*}
	\displaystyle\sum_{2\leq m \leq (\log T)^{2\gamma}}\dfrac{\Lambda(m)}{(\log m)^{n}m^\sigma}\ll \dfrac{(\log T)^{2\gamma(1-\sigma)}}{(\log_2T)^n}, 
	\end{equation*}
	and this implies the desired result.
\end{proof}

\section{Proof of Theorem \ref{27_9_7:50am}} Assume the Riemann hypothesis and consider the parameters defined in Proposition \ref{9_21_4:20pm}, Subsection \ref{subsection}. Throughout this section we will assume that
\begin{align} \label{9_30_3:27pm}
0\leq h \leq \dfrac{1}{2\log_2T}.
\end{align}
Using the fact that $\sin(x)\gg x$ for $0\leq x\leq1$, we obtain the bound
\begin{align} \label{9_26_6:25pm}
\sin(h\log m) \gg h\log m
\end{align} 
for $m\leq (\log T)^2$. We integrate \eqref{22_9_29_12} in the range $T^\beta\leq t\leq T\log T$  with $|R(t)|^2\,\Phi(t/T)$, and by $(ii)$ of Proposition \ref{19_4_6:04pm} we get
\begin{align} \label{9_22_2:12pm}
\begin{split}
&\int_{T^\beta}^{T\log T}|R(t)|^2\Phi\bigg(\dfrac{t}{T}\bigg)\Bigg(\int_{-(\log T)^{3}}^{(\log T)^{3}} \Delta_hS_n(\sigma,t+u)\bigg(\dfrac{\sin(\gamma u\log_2T)}{u}\bigg)^2\big(3\,\delta+2\,\delta'\sin(u\log_2T)\big)\du\Bigg)\dt  \\
& = 3\,\delta\,\im\Bigg\{i^{n+3}\displaystyle\sum_{m=2}^\infty\dfrac{\Lambda(m)\,a_m(T,h)}{(\log m)^{n+1}m^\sigma}\bigg(\int_{T^\beta}^{T\log T}m^{-it}|R(t)|^2\Phi\bigg(\dfrac{t}{T}\bigg)\dt\bigg)\Bigg\}   \\
& \,\,\,\,\,\,+\delta'\,\im\Bigg\{i^{n+2}\displaystyle\sum_{m=2}^\infty\dfrac{\Lambda(m)\,b_m(T,h)}{(\log m)^{n+1}m^\sigma}\bigg(\int_{T^\beta}^{T\log T}m^{-it}|R(t)|^2\Phi\bigg(\dfrac{t}{T}\bigg)\dt\bigg)\Bigg\}  + O\bigg(h\,T\log_2T\displaystyle\sum_{l\in\mathcal{M}}f(l)^2\bigg).
\end{split}
\end{align}
We want to complete the integrals that appears on the right-hand side of \eqref{9_22_2:12pm}, from $0$ to $\infty$. Using the estimate $|a_m(T,h)|\ll h\log m\log_2T$, $(i)$ of Proposition \ref{19_4_6:04pm} and the bound $\Phi(t)\leq 1$, we have
\begin{align*} 
\Bigg|\displaystyle\sum_{m=2}^\infty\dfrac{\Lambda(m)\,a_m(T,h)}{(\log m)^{n+1}m^\sigma}&\bigg(\int_{0}^{T^\beta}m^{-it}|R(t)|^2\Phi\bigg(\dfrac{t}{T}\bigg)\dt\bigg)\Bigg| \ll h\,T^{\kappa+\beta}\log_2T\displaystyle\sum_{l\in\mathcal{M}}f(l)^2\Bigg(\displaystyle\sum_{m\leq (\log T)^{2\gamma}}\dfrac{\Lambda(m)}{(\log m)^{n}m^{\sigma}}\Bigg).
\end{align*}
Therefore, using the prime number theorem we get 
\begin{align*} 
\Bigg|\displaystyle\sum_{m=2}^\infty\dfrac{\Lambda(m)\,a_m(T,h)}{(\log m)^{n+1}m^\sigma}\bigg(\int_{0}^{T^\beta}m^{-it}|R(t)|^2\Phi\bigg(\dfrac{t}{T}\bigg)\dt\bigg)\Bigg| \ll h\,T^{\kappa+\beta}\dfrac{(\log T)^{2\gamma(1-\sigma)}}{(\log_2T)^{n-1}}\displaystyle\sum_{l\in\mathcal{M}}f(l)^2 \ll h\,T\displaystyle\sum_{l\in\mathcal{M}}f(l)^2.
\end{align*}
Similarly, using the decay of $\Phi(t)$ we obtain 
\begin{align*} 
\Bigg|\displaystyle\sum_{m=2}^\infty\dfrac{\Lambda(m)\,a_m(T,h)}{(\log m)^{n+1}m^\sigma}&\bigg(\int_{T\log T}^{\infty}m^{-it}|R(t)|^2\Phi\bigg(\dfrac{t}{T}\bigg)\dt\bigg)\Bigg| \ll h\,T\displaystyle\sum_{l\in\mathcal{M}}f(l)^2.
\end{align*}
The analysis for $b_m(T,h)$ is analogous. Therefore, we can extend the integrals on the right-hand side of \eqref{9_22_2:12pm} from  $0$ to $\infty$. Now, we want to estimate the left-hand side of \eqref{9_22_2:12pm}. Assume that\footnote{\,\,\, In fact, the positivity of the right-hand side of \eqref{9_22_2:12pm} will be proved in the following subsections.} 
$$
\displaystyle\max_{\frac{T^\beta}{2}\leq t\leq 2T\log T}  \delta\,\Delta_hS_n(\sigma,t)>0.
$$
Using $(ii)$ of Proposition \ref{19_4_6:04pm}, it follows that
\begin{align*} 
&\int_{T^\beta}^{T\log T}|R(t)|^2\Phi\bigg(\dfrac{t}{T}\bigg)\Bigg(\int_{-(\log T)^{3}}^{(\log T)^{3}} \Delta_h S_n(\sigma,t+u)\bigg(\dfrac{\sin(\gamma u\log_2T)}{u}\bigg)^2\big(3\,\delta+2\,\delta'\sin(\gamma'u\log_2T)\big)\du\Bigg)\dt \nonumber \\
&\leq \log_2T\,\Bigg(\displaystyle\max_{\frac{T^\beta}{2}\leq t\leq 2T\log T}  \delta\,\Delta_hS_n(\sigma,t)\Bigg)\int_{-(\log T)^{3}\log_2T}^{(\log T)^{3}\log_2T} \bigg(\dfrac{\sin(\gamma u)}{u}\bigg)^2\bigg( 3+\dfrac{2\,\delta'}{\delta}\sin(\gamma'u)\bigg)\du\int_{T^\beta}^{T\log T}|R(t)|^2\Phi\bigg(\dfrac{t}{T}\bigg)\dt \nonumber \\
&\ll \log_2T\,\Bigg(\displaystyle\max_{\frac{T^\beta}{2}\leq t\leq 2T\log T}  \delta\,\Delta_hS_n(\sigma,t)\Bigg)\int_{-\infty}^{\infty} \bigg(\dfrac{\sin(\gamma u)}{u}\bigg)^2\du\int_{0}^{\infty}|R(t)|^2\Phi\bigg(\dfrac{t}{T}\bigg)\dt \nonumber \\
& \ll \Bigg(\displaystyle\max_{\frac{T^\beta}{2}\leq t\leq 2T\log T} \delta\,\Delta_hS_n(\sigma,t)\Bigg)\,T\log_2T\displaystyle\sum_{l\in\mathcal{M}}f(l)^2.
\end{align*}
Therefore, we obtain the following relation from \eqref{9_22_2:12pm}
\begin{align} \label{9_22_2:26pm} 
\begin{split}
 3\,\delta&\,\im\Bigg\{i^{n+3}\displaystyle\sum_{m=2}^\infty\dfrac{\Lambda(m)\,a_m(T,h)}{(\log m)^{n+1}m^\sigma}\bigg(\int_{0}^{\infty}m^{-it}|R(t)|^2\Phi\bigg(\dfrac{t}{T}\bigg)\dt\bigg)\Bigg\}   \\
& +\delta'\,\im\Bigg\{i^{n+2}\displaystyle\sum_{m=2}^\infty\dfrac{\Lambda(m)\,b_m(T,h)}{(\log m)^{n+1}m^\sigma}\bigg(\int_{0}^{\infty}m^{-it}|R(t)|^2\Phi\bigg(\dfrac{t}{T}\bigg)\dt\bigg)\Bigg\} + O\bigg(h\,T\log_2T\displaystyle\sum_{l\in\mathcal{M}}f(l)^2\bigg) \\
&\,\,\,\,\,\,\,\,\,\,\,\,\,\,\,\,\,\ll \Bigg(\displaystyle\max_{\frac{T^\beta}{2}\leq t\leq 2T\log T} \delta\,\Delta_hS_n(\sigma,t)\Bigg)\,T\log_2T\displaystyle\sum_{l\in\mathcal{M}}f(l)^2.
\end{split}
\end{align}Let us to analyze the left-hand side of \eqref{9_22_2:26pm}.\subsection{The case $n\,\equiv\, 1 \, (\mathrm{mod} \, 2)$} We choose the parameters $\gamma=1/8$, $\delta\in\{-1,1\}$ and $\delta'=(-1)^{(n+1)/2}$. Using the fact that $i^{n+2}=(-1)^{(n+1)/2}\,i$, we conclude from \eqref{9_22_2:26pm} that
\begin{align*}  &3\,\delta(-1)^{(n+3)/2}\,\im\Bigg\{\displaystyle\sum_{m=2}^\infty\dfrac{\Lambda(m)\,a_m(T,h)}{(\log m)^{n+1}m^\sigma}\bigg(\int_{0}^{\infty}m^{-it}|R(t)|^2\Phi\bigg(\dfrac{t}{T}\bigg)\dt\bigg)\Bigg\}   \\
& +\re\Bigg\{\displaystyle\sum_{m=2}^\infty\dfrac{\Lambda(m)\,b_m(T,h)}{(\log m)^{n+1}m^\sigma}\bigg(\int_{0}^{\infty}m^{-it}|R(t)|^2\Phi\bigg(\dfrac{t}{T}\bigg)\dt\bigg)\Bigg\} + O\bigg(h\,T\log_2T\displaystyle\sum_{l\in\mathcal{M}}f(l)^2\bigg) 
\\
& \,\,\,\,\,\,\,\,\,\,\,\,\,\,\,\,\,\,\,\,\,\,\,\,\,\,\,\,\,\,\,\,\,\ll \Bigg(\displaystyle\max_{\frac{T^\beta}{2}\leq t\leq 2T\log T} \delta\,\Delta_hS_n(\sigma,t)\Bigg)\,T\log_2T\displaystyle\sum_{l\in\mathcal{M}}f(l)^2.
\end{align*}
Using the fact that $|R(t)|^2$ and $\Phi(t)$ are real and even functions, we have
\begin{align} \label{1_11_31_01}
\re\Bigg\{\int_{0}^{\infty}m^{-it}|R(t)|^2\Phi\bigg(\dfrac{t}{T}\bigg)\dt\Bigg\} = \dfrac{1}{2}\int_{-\infty}^{\infty}m^{-it}|R(t)|^2\Phi\bigg(\dfrac{t}{T}\bigg)\dt.
\end{align}
Therefore, by Lemma \ref{9_22_6:02pm2} it follows that
\begin{align}  \label{2_16_02_20}
\begin{split} 
\dfrac{1}{2}\displaystyle\sum_{m=2}^\infty\dfrac{\Lambda(m)\,b_m(T,h)}{(\log m)^{n+1}m^\sigma}&\bigg(\int_{-\infty}^{\infty}m^{-it}|R(t)|^2\Phi\bigg(\dfrac{t}{T}\bigg)\dt\bigg) + O\Bigg(h\,T\,(\log T)^{(1-\sigma)/4} \log_2T \displaystyle\sum_{l\in\mathcal M}f(l)^2\Bigg) \\
& \ll \Bigg(\displaystyle\max_{\frac{T^\beta}{2}\leq t\leq 2T\log T} \delta\,\Delta_hS_n(\sigma,t)\Bigg)\,T\log_2T\displaystyle\sum_{l\in\mathcal{M}}f(l)^2.
\end{split}
\end{align}
Note that the sum on the above expression runs over $(\log T)^{3/4}\leq m\leq (\log T)^{5/4}$. Then, for $(\log T)^{7/8}\leq m\leq (\log T)^{9/8}$ we have that  $w_m(\log_2T/8,\log_2T)\gg \log m$. Therefore, using \eqref{9_26_6:25pm} we conclude $b_m(T,h) \gg h (\log m)^2$ for $(\log T)^{7/8}\leq m\leq (\log T)^{9/8}$. Using  \eqref{20_4_6:46pm}, for each $p\in \mathcal{P}$, we have $b_p(T,h)\gg h(\log p)^2$, for $T$ sufficiently large. This implies that 
$$
\min_{p\in\mathcal{P}}\dfrac{b_p(T,h)}{(\log p)^n}\gg \min_{p\in\mathcal{P}} \dfrac{h}{(\log p)^{n-2}} \gg \dfrac{h}{(\log_2T)^{n-2}}.
$$
Then, using Lemma \ref{20_41_1_20} we have
$$
\displaystyle\sum_{m=2}^\infty\dfrac{\Lambda(m)\,b_m(T,h)}{(\log m)^{n+1}m^\sigma}\bigg(\int_{-\infty}^{\infty}m^{-it}|R(t)|^2\Phi\bigg(\dfrac{t}{T}\bigg)\dt\bigg)\gg h\,T\,\dfrac{(\log T)^{1-\sigma}(\log_3T)^{\sigma}}{(\log_2T)^{\sigma+n-2}}\displaystyle\sum_{l\in\mathcal{M}}f(l)^2.
$$
Inserting this estimate in \eqref{2_16_02_20}, it follows that
\begin{equation} \label{16_8_31_12}
h\,\dfrac{(\log T)^{1-\sigma}(\log_3T)^{\sigma}}{ (\log_2T)^{\sigma+n-1}}\ll \displaystyle\max_{\frac{T^\beta}{2}\leq t\leq 2T\log T} \delta\,\Delta_hS_n(\sigma,t),
\end{equation}
for any $\delta\in\{-1,1\}$ and $h$ satisfying \eqref{9_30_3:27pm}. Since that \eqref{17_24} holds, we can change the left-hand side of \eqref{16_8_31_12} by
$$h\,\dfrac{(\log T)^{1/2}(\log_3T)^{1/2}}{ (\log_2T)^{n-1/2}}.
$$
We replace $\Delta_hS_n(\sigma,t)$ with $S_n(\sigma,t+2h)-S_n(\sigma,t)$, by changing $t-h$ to $t$, where the maximum is taken over $T^{\beta}/3\leq t\leq 3T\log T$. We obtain the desired result after a trivial adjustment, changing $T$ to $T/3\log T$ and choosing a slightly smaller $\beta$. 
%Note that we obtain the desired result for $S_n(\sigma,t+h)-S_n(\sigma,t)$ when $h\in [0,(\log_2T)^{-1}]$.

\subsection{The case $n\,\equiv\, 0 \, (\mathrm{mod} \, 2)$}
We choose the parameters $\gamma=2/3$, $\delta=(-1)^{(n+2)/2}$ and $\delta'=0$. Using the fact that $i^{n+3}=(-1)^{(n+2)/2}\,i$, we conclude from \eqref{9_22_2:26pm} that
\begin{align} \label{9_27_9:11am}
\begin{split}
&3\,\re\Bigg\{\displaystyle\sum_{m=2}^\infty\dfrac{\Lambda(m)\,a_m(T,h)}{(\log m)^{n+1}m^\sigma}\bigg(\int_{0}^{\infty}m^{-it}|R(t)|^2\Phi\bigg(\dfrac{t}{T}\bigg)\dt\bigg)\Bigg\}  + O\bigg(h\,T\log_2T\displaystyle\sum_{l\in\mathcal{M}}f(l)^2\bigg) \\
&\ll \Bigg(\displaystyle\max_{\frac{T^\beta}{2}\leq t\leq 2T\log T} \delta\,\Delta_hS_n(\sigma,t)\Bigg)\,T\log_2T\displaystyle\sum_{l\in\mathcal{M}}f(l)^2.
\end{split}
\end{align}
Therefore, using \eqref{1_11_31_01}, Lemma \ref{20_41_1_20} and doing the same procedure as in the previous case, we obtain the required lower bound.

\section{Proof of theorem \ref{27_9_7:50am2}}
The proof for the case of $S(t)$ follows the same outline of Theorem \ref{27_9_7:50am}, but with a slight change in Lemma  \ref{9_20_10:05am}. By \cite[Lemma 2]{Ma}, we have that
	\begin{align*}
	\int_{-(\log T)^{3}}^{(\log T)^{3}}\log\zeta\bigg(\dfrac{1}{2}+i(t+u)\bigg)&\bigg(\dfrac{\sin\alpha u}{u}\bigg)^2e^{iHu}\du = \dfrac{\pi}{2}\displaystyle\sum_{m=2}^{\infty}\dfrac{\Lambda(m)\,w_m(\alpha,H)}{(\log m)\,m^{\frac{1}{2}+it}} + O\bigg(\dfrac{e^{2\alpha + |H|}}{(\log T)^3}\bigg).
	\end{align*} 
Therefore for $0\leq h\leq 1$ and for sufficiently large $T$, we have
$$
\int_{-(\log T)^{3}}^{(\log T)^{3}}\Delta_h\log\zeta\bigg(\dfrac{1}{2}+i(t+u)\bigg)\bigg(\dfrac{\sin\alpha u}{u}\bigg)^2e^{iHu}\du =-\pi i\displaystyle\sum_{m=2}^{\infty}\dfrac{\Lambda(m)\,w_m(\alpha,H)\sin(h\log m)}{(\log m)\,m^{\sigma+it}} + O\bigg(\dfrac{e^{2\alpha + |H|}}{(\log T)^3}\bigg).
$$
Note that the main difference of this estimate from the Lemma \ref{9_20_10:05am} appears in the error term. Computing exactly as in the preceding cases, we get the following equivalent formula for the equation \eqref{9_27_9:11am},
\begin{align*} %\label{9_27_9:11am2}
3\,\re\Bigg\{\displaystyle\sum_{m=2}^\infty\dfrac{\Lambda(m)\,a_m(T,h)}{(\log m)m^\frac{1}{2}}&\bigg(\int_{0}^{\infty}m^{-it}|R(t)|^2\Phi\bigg(\dfrac{t}{T}\bigg)\dt\bigg)\Bigg\}  + O\bigg(T\log_2T\displaystyle\sum_{l\in\mathcal{M}}f(l)^2\bigg) \\
&\ll \Bigg(\displaystyle\max_{\frac{T^\beta}{2}\leq t\leq 2T\log T} -\,\Delta_hS(t)\Bigg)\,T\log_2T\displaystyle\sum_{l\in\mathcal{M}}f(l)^2.
\end{align*}
Then, by \eqref{1_11_31_01} and Lemma \ref{20_41_1_20} we obtain 
\begin{align*}
h\,(\log T)^{1/2}(\log_2T)^{1/2}(\log_3T)^{1/2} + O(1) \ll\displaystyle\max_{\frac{T^\beta}{2}\leq t\leq 2T\log T} -\,\Delta_hS(t).
\end{align*}
Here appears the new restriction for $h$. Choosing $h\in [c_1(\log T)^{-1/2}(\log_2 T)^{-1/2}(\log_3 T)^{-1/2},(\log_2 T)^{-1}]$, for a suitable constant $c_1$, and adjusting $T$ and $\beta$, we obtain the desired result.

\medskip

\section{Proof of the corollaries}

\subsection{Proof of Corollary \ref{9_29_4:06pm}} The proof follows from the following inequality for $n\geq 0$,
	\begin{align*} % \label{9_29_4:14pm}
	\displaystyle\max_{u\in [t,t+h]}\pm{S_n(\sigma,u)}\geq h^{-1}\int_{t}^{t+h}\pm S_n(\sigma,u)\du = h^{-1}\left(\pm\{S_{n+1}(\sigma,t+h)-S_{n+1}(\sigma,t)\}\right),
	\end{align*}
	with $h=(\log\log T)^{-1}$, changing $T$ to $T/(2\log\log T)$ and making $\beta$ slightly smaller.

\subsection{Proof of Corollary \ref{21_322}}
	Under RH, it follows from Corollary \ref{9_29_4:06pm}. If the Riemann hypothesis fails, by \cite[Page 6]{F1}, we have
	$$
	S_n(t)\gg t^{n-2},
	$$
	for $t$ sufficiently large. This implies the desired result.

\medskip
\section*{Acknowledgments}
We would like to thank Kristian Seip for many valuable discussions and for their insightful comments. We would also like to thank Micah B. Milinovich for pointing to us to use the result of Fujii to obtain Corollary \ref{21_322}, and the anonymous referee for the careful review.

\end{document}